\UseRawInputEncoding
\documentclass[leqno,12pt]{amsart}
\usepackage{amsmath,amstext,amssymb,amsopn,amsthm,mathrsfs}
\usepackage{epsfig,graphics,latexsym,graphicx,color}
\issueinfo{}
\copyrightinfo{}
  {}
\pagespan{1}{}
\usepackage{graphicx}
\allowdisplaybreaks
\definecolor{darkgreen}{rgb}{0,.5,0}
\numberwithin{equation}{section}
\allowdisplaybreaks

\newtheorem{theorem}{Theorem}[section]
\newtheorem{lemma}{Lemma}[section]

\newtheorem{definition}{Definition}[section]

\newtheorem*{rem*}{Remark}
\newtheorem{remark}{Remark}[section]

\begin{document}
\title[]{Boundedness of some operators on grand generalized weighted Morrey spaces on RD-spaces}

\author[]{Suixin He}
\address{College of Mathematics and Statistics, Northwest Normal University, Lanzhou 730070, China}
\email{hesuixinmath@126}

\author[]{Shuangping Tao$^{\ast}$}
\address{College of Mathematics and Statistics, Northwest Normal University, Lanzhou 730070, China}
\email{taosp@nwnu.edu.cn}
\thanks{This research was supported by National Natural Science Foundation of China(Grant No.11561062)}



\begin{abstract}
The aim of this paper is to obtain the boundedness of some operator on grand generalized weighted Morrey spaces $\mathcal{L}^{p),\phi}_{\varphi}(\omega)$  over RD-spaces. Under assumption that functions $\varphi$ and $\phi$ satisfy certain conditions, the authors prove that Hardy-Littlewood maximal operator and $\theta$-type Calder\'{o}n-Zygmund operator are bounded on grand generalized weighted Morrey spaces $\mathcal{L}^{p),\phi}_{\varphi}(\omega)$. Moreover, the boundedness of commutator $[b,T_{\theta}]$ which is generated by $\theta$-type Calder\'{o}n-Zygmund operator $T_{\theta}$ and $b\in\mathrm{BMO}(\mu)$ on spaces $\mathcal{L}^{p),\phi}_{\varphi}(\omega)$ is also established. The results regarding the grand generalized weighted Morrey spaces is new even for domains of Euclidean spaces.
\end{abstract}
\maketitle

\section{Introduction}
The study of the spaces of homogeneous type, first introduced by Coifman and Weiss \cite{CW1,CW2}, is a general framework for studying the Calder\'on-Zygmund operators and functions spaces. Around 1970s, Coifman and Weiss started to investigate the some harmonic analysis problems on the metric spaces called space of homogeneous type $(X,d,\mu)$ equipped with a metric $d$ and a regular Borel measure $\mu$ satisfying the doubling condition, i.e., if there exists a positive constant $C_0>1$ such that, for any ball $B(x,r):=\{y\in X:d(x,y)<r\}$ with $x\in X$ and $r>0$ ,
\begin{eqnarray}
\mu(B(x,2r))\leq C_0\mu(B(x,r))
\end{eqnarray}
holds. Since then, many experts have extended some classical results to spaces of homogeneous type in the sense of Coifman and Weiss. However, some harmonic analysis results have so far obtained only on the RD-spaces, which
means that $(X,d,\mu)$ is a space of homogeneous type if there exists positive constants $a, b>1$ such that,
\begin{eqnarray}
b\mu(B(x,r))\leq \mu(B(x,ar));
\end{eqnarray}
holds for all $x\in X$ and $r\in (0,\rm diam(X)/a)$. On the development and research of the operators over RD-spaces, we refer readers see, e.g. \cite{MY,YZ1,YZ2}.

Morrey spaces were introduced in 1938 by Morrey \cite{M} in relation to local regularity problems of solutions of the second order elliptic partial differential equations. In 2009, Komori and Shirai \cite{KS} introduced the weighted Morrey spaces on the Euclidean space. Since 2000, there are many papers focusing on Morrey spaces and weighted Morrey spaces on different setting, see,e.g.,\cite{AX,DR,HS,N,ST,H}.
The generalized weighted Morrey spaces over RD-spaces were introduced in \cite{CL},  where boundedness was established for the
Hardy-Littlewood maximal operator and Calder\'{o}n-Zygmund operator. Very recently, the boundedness of commutators the generalized by the $\theta$-Calder\'{o}n-Zygmund operator and the BMO functions in generalized weighted Morrey spaces over RD-spaces was already treated by Li \textit{et.~al}. \cite{LLW}.

Nowadays the theory of grand Lebesgue space introduced by Iwaniec and Sbordone \cite{IS} is one of the intensively developing directions in Modern analysis. It was realized the usefull in applications to partial differential equations, in geometric function theory, Sobolev spaces theory, see \cite{DSS,FFR,FLS,GIS}. Since then, Some classical operator of harmonic analysis have been intensively studied in recent years. For instance,
Kokilashvili \cite{K} established criteria for the boundedness of several well-known operators in the generalized weighted grand Lebesgue space.
In 2019, Kokilashvili \textit{et.~al} established the weighted extrapolation results in grand Morrey spaces and obtained some applications in PDE \cite{KMR}. Recently, the authors \cite{HT} have obtained the boundedness of some operator on grand generalized Morrey space over non-homogeneous spaces. The more research on the boundedness of operators in grand spaces, we mention e.g. \cite{KM1,KM2,KMR1,MS} and references therein.

Inspired by the above studies, in this paper, we will establish the boundedness Hardy-Littlewood maximal operator and $\theta$-type Calder\'{o}n-Zygmund operators on grand generalized weighted Morrey space over RD-spaces.
For the study of maximal operators and $\theta$-type Calder\'on-Zygmund operators in grand generalized weighted Morrey space defined on RD-spaces, we depend on the results of references \cite{CL,LLW}.

Let $1<p<\infty$ and $\varphi$ be a function on $(0,p-1]$ which is a positive bounded and satisfies $\mathop{\lim}\limits_{x\rightarrow0}\varphi(x)=0$. The class of such functions will be simply denoted by $\Phi_{p}$.  Then the norm of functions $f$ in weighted grand Lebesgue space ${L}^{p)}_{\varphi}(\omega)$ is defined by
\begin{eqnarray}
\|f\|_{{L}^{p)}_{\varphi}(\omega)}=\mathop{\sup}\limits_{0<\varepsilon<p-1}[\varphi(\varepsilon)]^{\frac{1}{p-\varepsilon}}\|f\|_{L^{p-\varepsilon}(\omega)},
\end{eqnarray}
where $L^{r}(\omega)$ is the classical Lebesgue space with respect to a measure $\mu$, and defined by the norm:
$$
\|f\|_{L^{r}(\omega)}:=\bigg(\int_{X}|f(x)|^{r}\omega(x)\mathrm{d}\mu(x)\bigg)^{\frac{1}{r}},\quad 1\leq r<\infty.
$$

On the base of weighted grand Lebesgue space ${L}^{p)}_{\varphi}(\omega)$ , we give the definition of grand generalized weighted Morrey spaces as follows.
\begin{definition}\rm
(Grand generalized weighted Morrey spaces)
Let $1<p<\infty$, let $\omega$ be a weight and $\varphi\in\Phi_{p}$. Suppose that $\phi: (0,\infty)\rightarrow(0,\infty)$ is an increasing function. Then grand generalized weighted Morrey space
$\mathcal{L}^{p),\phi}_{\varphi}(\omega)$ is defined by
$$
\|f\|_{\mathcal{L}^{p),\phi}_{\varphi}(\omega)}:=\bigg\{f\in L^{1}_{\mathrm{loc}}(\omega): \|f\|_{\mathcal{L}^{p),\phi}_{\varphi}(\omega)}<\infty\bigg\},
$$
where
\begin{eqnarray}
&&\|f\|_{\mathcal{L}^{p),\phi}_{\varphi}(\omega)}\\\nonumber
&&\quad:=\mathop{\sup}\limits_{0<\varepsilon<p-1}\varphi(\varepsilon)
\mathop{\sup}\limits_{B\subset X}
[\phi(\omega(B))]^{-\frac{1}{p-\varepsilon}}\bigg(\int_{B}|f(x)|^{p-\varepsilon}\omega(x)\mathrm{d}\mu(x)\bigg)^{\frac{1}{p-\varepsilon}}\\\nonumber
&&\quad\ =\mathop{\sup}\limits_{0<\varepsilon<p-1}\varphi(\varepsilon)\|f\|_{\mathcal{L}^{p-\varepsilon,\phi}(\omega)}.
\end{eqnarray}
\end{definition}

Especially, if we take $\varphi(\varepsilon)=\varepsilon^{\theta}$ with $\theta>0$ in (1.4), then we can denote
$$
\|f\|_{\mathcal{L}^{p),\phi}_{\varphi}(\omega)}:=\|f\|_{\mathcal{L}^{p),\phi}_{\theta}(\omega)}.
$$
\begin{remark}\rm
$\mathrm{(1)}$\ When $\phi(x)=1$, $\mathcal{L}^{p),\phi}_{\varphi}(\omega)=L^{p)}_{\varphi}(\omega)$. Therefore, the grand generalized weighted Morrey space $\mathcal{L}^{p),\phi}_{\varphi}(\omega)$ is an extension of the grand weighted Lebesgue space.

$\mathrm{(2)}$\ If $\omega\in A_p(\mu)$, the generalized weighted Morrey space $\mathcal{L}^{p,\phi}(\omega)$ (see \cite{CL}), which is defined with respect to the norm:
\begin{eqnarray}
\|f\|_{\mathcal{L}^{p,\phi}(\omega)}:=\mathop{\sup}\limits_{B\subset X}\bigg(\frac{1}{\phi(\omega(B))}\int_{B}|f(x)|^{p}\omega(x)\mathrm{d}\mu(x)\bigg)^{\frac{1}{p}},\quad 1\leq p<\infty.
\end{eqnarray}

$\mathrm{(3)}$\ If we take  function $\phi(t)=t^{\frac{p}{q}-1}$ for $t>0$ and $1<p\leq q<\infty$, then grand generalized weighted Morrey space $\mathcal{L}^{p),\phi}_{\varphi}(\omega)$ defined as in $\mathrm{(1.4)}$ is just the grand weighted Morrey space $\mathcal{L}^{p),q}_{\varphi}(\omega)$ which is sightly modified in \cite{KMR}, that is,
\begin{eqnarray}
\|f\|_{\mathcal{L}^{p),q}_{\varphi}(\omega)}=\mathop{\sup}\limits_{0<\varepsilon<p-1}\varphi(\varepsilon)\mathop{\sup}\limits_{B}
[\omega(B)]^{\frac{1}{q}-\frac{1}{p-\varepsilon}}\|f\|_{L^{p-\varepsilon}(\omega)}.
\end{eqnarray}
\end{remark}

\begin{definition}\rm
We say that a weight function $\omega$ belongs to the Muckenhoupt class $A_{p}(\mu)$, $1<p<\infty$, if
$$
\|\omega\|_{A_{p}}:=\mathop{\sup}\limits_{B}\bigg(\frac{1}{\mu(B)}\int_{B}\omega(x)\mathrm{d}\mu(x)\bigg)\bigg(\frac{1}{\mu(B)}\int_{B}[\omega(x)]^{1-p^{\prime}}\mathrm{d}\mu(x)\bigg)^{p-1}<\infty,
$$
where the supremum is taken over all balls $B\subset X$.

Further, $\omega\in A_{1}(\mu)$ if there is a positive constant $C$ such that, for any ball $B\subset X$,
$$\frac{1}{\mu(B)}\int_{B}\omega(x)\mathrm{d}\mu(x)\leq C \mathop{ess\inf}\limits_{y\in B}\omega(y),$$
as in the classical setting, let $A_{\infty}(\mu)=\bigcup^{\infty}_{p=1}A_{p}(\mu)$.
\end{definition}

\hspace{-4mm}{\bf Notation}

$\bullet$  $C$ represents a positive constant which is independent of the main parameters;

$\bullet$  $p^{\prime}$ stands for the conjugate exponent $\frac{1}{p}+\frac{1}{p^{\prime}}=1$;

$\bullet$  $B(x,r)=\{y\in X: d(x,y)<r\}$;

$\bullet$  For any $x,y\in X$ and $\delta\in(0,\infty)$, Let $V(x,y):=\mu(B(x,d(x,y)))$ and $V_{\delta}:=\mu(B(x,\delta))$, it follows from doubling condition that $V(x,y)=V(y,x)$.

Throughout the paper we assume that $\mu(X)<\infty$,

\section{Hardy-Littlewood maximal operator on $\mathcal{L}^{p),\phi}_{\varphi}(\omega)$}

\textbf{\bf 2.1. Weighted boundedness of the maximal operator.} In this subsection we study the one-weighted problem for the Hardy-Littlewood maximal function $M$ defined by setting
\begin{eqnarray}
Mf(x):=\mathop{\sup}\limits_{r>0}\frac{1}{\mu(B(x,r))}\int_{B(x,r)}|f(y)|\mathrm{d}\mu(y),~~ \rm for~ all~x\in X
\end{eqnarray}

\begin{lemma}\rm \cite{ST1} Let $p\in(1,\infty)$ and $\omega\in A_{p}(\mu)$. There exist positive constants $C_1$ and $C_{2}$ such that for any ball $B\subset X$ and each measurable set $E\subseteq B$,
\begin{eqnarray*}
\frac{\omega(E)}{\omega(B)}\leq C_1 \bigg[\frac{\mu(E)}{\mu(B)}\bigg]^{\frac{1}{p}} \quad and  \quad \frac{\omega(E)}{\omega(B)}\geq C_2 \bigg[\frac{\mu(E)}{\mu(B)}\bigg]^{p}.
\end{eqnarray*}
\end{lemma}

\begin{lemma}\rm \cite{CL} Let $(X,d,\mu)$ be an RD-space, if $\omega\in A_{p}(\mu)$, $p\in(1,\infty)$, then there exist positive constants $C_3, C_{4}>1$ such that for any ball $B\subset X$,
\begin{eqnarray}
\omega(2B)\geq C_3  \omega(B),
\end{eqnarray}
\begin{eqnarray}
\omega(2B)\leq C_4  \omega(B).
\end{eqnarray}
\end{lemma}

\begin{lemma}\rm \cite{CL}
Let $1<p<\infty$, $\omega\in A_{p}(\mu)$, $\varphi\in\Phi_{p}$ and $\phi: (0,\infty)\rightarrow(0,\infty)$ be an increasing function. Assume that
the mapping $t\mapsto\frac{\phi(t)}{t}$ is almost decreasing.
Then $M$ be as in (2.1) is bounded on $\mathcal{L}^{p,\phi}(\omega)$.
\end{lemma}

\begin{theorem}\rm \ \
Let $1<p<\infty$, $\omega\in A_{p}(\mu)$, $\varphi\in\Phi_{p}$ and $\phi: (0,\infty)\rightarrow(0,\infty)$ be an increasing function. Let $M$ be as in (2.1). Assume that
the mapping $t\mapsto\frac{\phi(t)}{t}$ is almost decreasing, namely, there exists a positive constant $C$ such that
\begin{eqnarray}
\frac{\phi(t)}{t}\leq C\frac{\phi(s)}{s},
\end{eqnarray}
for $s\geq t$.
Then there exists a positive constant $C$ such that for any $f\in \mathcal{L}^{p),\phi}_{\varphi}(\omega)$,
\begin{eqnarray*}
\|M(f)\|_{\mathcal{L}^{p),\phi}_{\varphi}(\omega)}\leq C\|f\|_{\mathcal{L}^{p),\phi}_{\varphi}(\omega)}
\end{eqnarray*}
\end{theorem}

\begin{proof}[\textbf{Proof}]
Choosing a number $\delta$ such that $0<\varepsilon\leq \delta<p-1$, observe that
\begin{eqnarray*}
\|M(f)\|_{\mathcal{L}^{p),\phi}_{\varphi}(\omega)}&&=\mathop{\sup}\limits_{0<\varepsilon<p-1}\varphi(\varepsilon)\mathop{\sup}\limits_{B\subset X}
[\phi(\omega(B))]^{-\frac{1}{p-\varepsilon}}\|M(f)\|_{L^{p-\varepsilon}(\omega)}\\
&&\quad\leq\mathop{\sup}\limits_{0<\varepsilon\leq \delta}\varphi(\varepsilon)\mathop{\sup}\limits_{B}
[\phi(\omega(B))]^{-\frac{1}{p-\varepsilon}}\|M(f)\|_{L^{p-\varepsilon}(\omega)}\\
&&\quad\quad+\mathop{\sup}\limits_{\delta<\varepsilon<p-1}\varphi(\varepsilon)\mathop{\sup}\limits_{B}
[\phi(\omega(B))]^{-\frac{1}{p-\varepsilon}}\|M(f)\|_{L^{p-\varepsilon}(\omega)}\\
&&\quad=:\mathrm{E}_{1}+\mathrm{E}_{2}.
\end{eqnarray*}

The estimates for $\mathrm{E}_{1}$ goes as follows. By applying the $\mathcal{L}^{p,\phi}(\omega)$-boundedness of $\mathrm{M}$ (see\cite{CL}) and (1.4), we can deduce that
\begin{eqnarray*}
&&\mathop{\sup}\limits_{0<\varepsilon\leq \delta}\varphi(\varepsilon)\mathop{\sup}\limits_{B}[\phi(\omega(B))]^{-\frac{1}{p-\varepsilon}}\|M(f)\|_{L^{p-\varepsilon}(\omega)}\\
&&\quad=\mathop{\sup}\limits_{0<\varepsilon\leq \delta}\varphi(\varepsilon)\|M(f)\|_{\mathcal{L}^{p,\phi}(\omega)}\\
&&\quad\leq C\|f\|_{\mathcal{L}^{p),\phi}_{\varphi}(\omega)}.
\end{eqnarray*}

Now let us estimate $\mathrm{E}_{2}$. Since $\delta<\varepsilon<p-1$, then we have $\frac{p-\delta}{p-\varepsilon}>1$. Further, by virtue of H\"{o}lder's inequality and $\mathcal{L}^{p,\phi}(\omega)$-boundedness of $M$, we get
\begin{eqnarray*}
\mathrm{E}_{2}&=&\mathop{\sup}\limits_{\delta<\varepsilon<p-1}\varphi(\varepsilon)\mathop{\sup}\limits_{B}
[\phi(\omega(B))]^{-\frac{1}{p-\varepsilon}}\|M(f)\|_{L^{p-\varepsilon}(\omega)}\\
&\leq&\mathop{\sup}\limits_{\delta<\varepsilon<p-1}\varphi(\varepsilon)\mathop{\sup}\limits_{B}
[\phi(\omega(B))]^{-\frac{1}{p-\varepsilon}}\|M(f)\|_{L^{p-\delta}(\omega)}(\omega(B))^{\frac{\varepsilon-\delta}{(p-\varepsilon)(p-\delta)}}\\
&=&\mathop{\sup}\limits_{\delta<\varepsilon<p-1}\varphi(\varepsilon)[\varphi(\delta)]^{-1}
\varphi(\delta)\mathop{\sup}\limits_{B}[\phi(\omega(B))]^{-\frac{1}{p-\varepsilon}}[\phi(\omega(B))]^{\frac{1}{p-\delta}}\\
&&\quad\quad\times[\phi(\omega(B))]^{-\frac{1}{p-\delta}}\|M(f)\|_{L^{p-\delta}(\omega)}(\omega(B))^{\frac{\varepsilon-\delta}{(p-\varepsilon)(p-\delta)}}\\
&=&\mathop{\sup}\limits_{\delta<\varepsilon<p-1}\varphi(\varepsilon)[\varphi(\delta)]^{-1}
\varphi(\delta)\mathop{\sup}\limits_{B}[\phi(\omega(B))]^{\frac{1}{p-\delta}-\frac{1}{p-\varepsilon}}[\omega(B)]^{\frac{\varepsilon-\delta}{(p-\varepsilon)(p-\delta)}}\\
&&\quad\quad\times[\phi(\omega(B))]^{-\frac{1}{p-\delta}}\|M(f)\|_{L^{p-\delta}(\omega)}\\
&\leq&C\|f\|_{\mathcal{L}^{p),\phi}_{\varphi}(\omega)}.
\end{eqnarray*}
Which, together with the estimate for $\mathrm{E}_{1}$, the {\bf Theorem 2.1} is proved.
\end{proof}

With an argument similar to that used in the proof of {\bf Theorem 2.1}, it is easy to obtain the following result on the maximal operator $\widetilde{M}_{r}$.\\
\textbf{Corollary 2.2.}\ \
Let $1<p<\infty$, $\omega\in A_{p}(\mu)$, $\varphi\in\Phi_{p}$ and $\phi: (0,\infty)\rightarrow(0,\infty)$ be an increasing function. Assume that
the mapping $t\mapsto\frac{\phi(t)}{t}$ is almost decreasing function satisfying $\mathrm{(2.4)}$. Then non-centered maximal operator $\widetilde{M}_{r}$ is bounded on $\mathcal{L}^{p),\phi}_{\varphi}(\omega)$, where $\widetilde{M}_{r}$ is defined by
$$
\widetilde{M}_{r}(f)(x):=\mathop{\sup}\limits_{x\in B}\bigg(\frac{1}{\mu(B)}\int_{B}|f(y)|^{r}\mathrm{d}\mu(y)\bigg)^{\frac{1}{r}}.
$$
\textbf{\bf 2.2. Vector-valued extension.} To discuss the vector-valued extension of {\bf Theorem 2.1}, we need the following assumption on $\phi$: there exists a positive constant
$C$ such that
\begin{eqnarray}
\int_{r}^{\infty}\frac{\phi(t)}{t}\frac{\mathrm{d}t}{t}\leq C\frac{\phi(r)}{r} ~~\rm for ~any ~r\in(0,\infty).
\end{eqnarray}
\begin{lemma}\rm \cite{CL}
Let $p\in(1,\infty)$, $\omega\in A_{p}(\mu)$ and $\phi: (0,\infty)\rightarrow(0,\infty)$ be an increasing function which satisfies (2.5), assume that the mapping $t\mapsto\frac{\phi(t)}{t}$ satisfies (2.4). Then there exists a positive constant $C$ such that for any ball $B\subset X$,
$$
\sum^{\infty}_{k=1}\bigg[\frac{\phi(\omega(2^{k}B))}{\omega(2^{k}B)}\bigg]^{\frac{1}{p}}\leq C \bigg[\frac{\phi(\omega(B))}{\omega(B)}\bigg]^{\frac{1}{p}}.$$
\end{lemma}

\begin{lemma}\rm \cite{FMY}
Let $r\in(1,\infty),p\in(1,\infty)$ and $\omega\in A_{p}(\mu)$. Then there exists a positive constant $C$, depending on $p$ and $r$, such that, for any
$\{f_i\}^{\infty}_{i=1}\subset L^{p}(\omega)$,
$$\bigg\|\bigg\{\sum_{j\in \mathbb{N}}\big[M(f_{j})\big]^{r}\bigg\}^{\frac{1}{r}}\bigg\|_{\mathcal{L}^{p}(\omega)}\leq C
\bigg\|\bigg\{\sum_{j\in \mathbb{N}}|f_j|^{r}\bigg\}^{\frac{1}{r}}\bigg\|_{\mathcal{L}^{p}(\omega)}.$$
\end{lemma}

\begin{theorem}\ \
Let $1<p,r<\infty$, $\omega\in A_{p}(\mu)$, $\varphi\in\Phi_{p}$ and $\phi: (0,\infty)\rightarrow(0,\infty)$ be an increasing function that satisfies (2.5). Let $M$ be as in (2.1). Assume that the mapping $t\mapsto\frac{\phi(t)}{t}$ satisfies (2.4). Then there exists a positive constant $C$, depending on $p$ and $r$, such that for any $\{f_j\}^{\infty}_{j=1}\subset \mathcal{L}^{p),\phi}_{\varphi}(\omega)$,
\begin{eqnarray*}
\bigg\|\bigg\{\sum_{j\in \mathbb{N}}\big[M(f_{j})\big]^{r}\bigg\}^{\frac{1}{r}}\bigg\|_{\mathcal{L}^{p),\phi}_{\varphi}(\omega)}\leq C
\bigg\|\bigg\{\sum_{j\in \mathbb{N}}|f_j|^{r}\bigg\}^{\frac{1}{r}}\bigg\|_{\mathcal{L}^{p),\phi}_{\varphi}(\omega)}.
\end{eqnarray*}
\end{theorem}
\begin{proof}[\textbf{Proof}]
Choosing a small $\delta$ such that $0<\varepsilon\leq \delta<p-1$, then, by applying {\bf Definition 1.1}, observe that
\begin{eqnarray*}
&&\bigg\|\bigg\{\sum_{j\in \mathbb{N}}\big[M(f_{j})\big]^{r}\bigg\}^{\frac{1}{r}}\bigg\|_{\mathcal{L}^{p),\phi}_{\varphi}(\omega)}\\
&&\quad\quad=
\mathop{\sup}\limits_{0<\varepsilon<p-1}\varphi(\varepsilon)\bigg\|\bigg\{\sum_{j\in \mathbb{N}}\big[M(f_{j})\big]^{r}\bigg\}^{\frac{1}{r}}\bigg\|_{\mathcal{L}^{p-\varepsilon,\phi}(\omega)}\\
&&\quad\quad\leq \mathop{\sup}\limits_{0<\varepsilon<\delta}\varphi(\varepsilon)\bigg\|\bigg\{\sum_{j\in \mathbb{N}}\big[M(f_{j})\big]^{r}\bigg\}^{\frac{1}{r}}\bigg\|_{\mathcal{L}^{p-\varepsilon,\phi}(\omega)}\\
&&\quad\quad\quad +\mathop{\sup}\limits_{\delta<\varepsilon<p-1}\varphi(\varepsilon)\bigg\|\bigg\{\sum_{j\in \mathbb{N}}\big[M(f_{j})\big]^{r}\bigg\}^{\frac{1}{r}}\bigg\|_{\mathcal{L}^{p-\varepsilon,\phi}(\omega)}\\
&&\quad\quad=:\mathrm{F}_{1}+\mathrm{F}_{2}.
\end{eqnarray*}
The estimates for $\mathrm{F}_{1}$ is given as follows. From {\bf Definition 1.1} and the $\mathcal{L}^{p,\phi}(\omega)-$boundedness of $\mathrm{M}$ (see\cite{CL}), it follows that
\begin{eqnarray*}
\mathop{\sup}\limits_{0<\varepsilon<\delta}\varphi(\varepsilon)\bigg\|\bigg\{\sum_{j\in \mathbb{N}}\big[M(f_{j})\big]^{r}\bigg\}^{\frac{1}{r}}\bigg\|_{\mathcal{L}^{p-\varepsilon,\phi}(\omega)}&\leq&
C\mathop{\sup}\limits_{0<\varepsilon<\delta}\varphi(\varepsilon)\bigg\|\bigg\{\sum_{j\in\mathbb{N}}|f_{j}|^{r}\bigg\}^{\frac{1}{r}}\bigg\|_{\mathcal{L}^{p-\varepsilon,\phi}(\omega)}\\
&\leq&C\bigg\|\bigg\{\sum_{j\in \mathbb{N}}|f_j|^{r}\bigg\}^{\frac{1}{r}}\bigg\|_{\mathcal{L}^{p),\phi}_{\varphi}(\omega)}.
\end{eqnarray*}

Similar to the estimate of $\mathrm{E}_{2}$ in the proof of {\bf Theorem 2.1}, By virtue of H\"{o}lder's inequality and {\bf Lemma 2.5}, we have
\begin{eqnarray*}
&&\mathop{\sup}\limits_{\delta<\varepsilon<p-1}\varphi(\varepsilon)\mathop{\sup}\limits_{B\subset X}[\phi(\omega(B))]^{-\frac{1}{p-\varepsilon}}
\bigg\|\bigg\{\sum_{j\in \mathbb{N}}\big[M(f_{j})\big]^{r}\bigg\}^{\frac{1}{r}}\bigg\|_{\mathcal{L}^{p-\varepsilon}(\omega)}\\
&&\quad \quad\quad \leq \mathop{\sup}\limits_{\delta<\varepsilon<p-1}\varphi(\varepsilon)\mathop{\sup}\limits_{B\subset X}[\phi(\omega(B))]^{-\frac{1}{p-\varepsilon}}
\bigg\|\bigg\{\sum_{j\in \mathbb{N}}\big[M(f_{j})\big]^{r}\bigg\}^{\frac{1}{r}}
\bigg\|_{\mathcal{L}^{p-\delta}(\omega)}(\omega(B))^{\frac{\varepsilon-\delta}{(p-\varepsilon)(p-\delta)}}\\
&&\quad \quad\quad \leq C \mathop{\sup}\limits_{\delta<\varepsilon<p-1}\varphi(\varepsilon)\mathop{\sup}\limits_{B\subset X}[\phi(\omega(B))]^{-\frac{1}{p-\varepsilon}}
\bigg\|\bigg\{\sum_{j\in \mathbb{N}}|f_{j}|^{r}\bigg\}^{\frac{1}{r}}
\bigg\|_{\mathcal{L}^{p-\delta}(\omega)}(\omega(B))^{\frac{\varepsilon-\delta}{(p-\varepsilon)(p-\delta)}}\\
&&\quad \quad\quad \leq C \mathop{\sup}\limits_{\delta<\varepsilon<p-1}\varphi(\varepsilon)[\varphi(\delta)]^{-1}\varphi(\delta)\mathop{\sup}\limits_{B\subset X}[\phi(\omega(B))]^{-\frac{1}{p-\varepsilon}}[\phi(\omega(B))]^{\frac{1}{p-\delta}}\\
&&\quad \quad\quad\quad \times[\phi(\omega(B))]^{-\frac{1}{p-\delta}}\bigg\|\bigg\{\sum_{j\in \mathbb{N}}|f_{j}|^{r}\bigg\}^{\frac{1}{r}}
\bigg\|_{\mathcal{L}^{p-\delta}(\omega)}(\omega(B))^{\frac{\varepsilon-\delta}{(p-\varepsilon)(p-\delta)}}\\
&&\quad \quad\quad \leq C \phi(p-1)[\phi(\delta)]^{-1}\bigg\|\bigg\{\sum_{j\in \mathbb{N}}|f_j|^{r}\bigg\}^{\frac{1}{r}}\bigg\|_{\mathcal{L}^{p),\phi}_{\varphi}(\omega)}\\
&&\quad \quad\quad \leq C\bigg\|\bigg\{\sum_{j\in \mathbb{N}}|f_j|^{r}\bigg\}^{\frac{1}{r}}\bigg\|_{\mathcal{L}^{p),\phi}_{\varphi}(\omega)}.\\
\end{eqnarray*}
Which, together with the estimate for $\mathrm{F}_{1}$, is our desired result.
\end{proof}

\section{$\theta$-Type  Calder\'{o}n-Zygmund operators on $\mathcal{L}^{p),\phi}_{\varphi}(\omega)$}
In this section deal with the boundedness of the $\theta$-type  Calder\'{o}n-Zygmund operators and its commutator on grand generalized weighted Morrey space $\mathcal{L}^{p),\phi}_{\varphi}(\omega)$ over RD-spaces.

The following definition see, Duong \textit{et.~al}. \cite{D}.
\begin{definition}\rm\ \
Let $\theta$ be a non-negative and non-decreasing function on $[0,\infty)$ with satisfying
\begin{eqnarray}
\int^{1}_{0}\frac{\theta(t)}{t}\mathrm{d}t<\infty.
\end{eqnarray}

And the measurable function  $K(\cdot,\cdot)$ on $X\times X\backslash\{(x,y): x\in X\}$ is called $\theta$-type kernel, if for any $x\neq y$,
\begin{eqnarray}
|K(x,y)|\leq \frac{C}{V(x,y)},
\end{eqnarray}
and for $d(x,z)<\frac{d(x,y)}{2}$,
\begin{eqnarray}
|K(x,y)-K(z,y)|+|K(y,x)-K(y,z)|\leq \frac{C}{V(x,y)}\theta\bigg(\frac{d(x,z)}{d(x,y)}\bigg).
\end{eqnarray}
\end{definition}

\begin{remark}\rm
If we take the function $\theta(t)=t^{\delta}$ with $t>0$ and $\delta \in (0,1]$. Then $K(x,y)$ defined as in Definition 3.1 is just the standard kernel.
\end{remark}

\begin{definition}\rm
Let $b$ a real valued $\mu-$measurable function on $X$, if $b \in L^{1}_{loc}(\mu)$ and its norm is
$$\|b\|_{\ast}:=\sup_{B}\frac{1}{\mu(B)}\int_{B}|b(x)-b_{B}|\mathrm{d}\mu(x)<\infty,$$
then $b$ is called a $\mathrm{BMO(\mu)}$ function, where the supremum is taken over all $B\subset X$ and $$b_B:=\frac{1}{\mu(B)}\int_{B}b(y)\mathrm{d}\mu(y).$$
\end{definition}
Let $L^{\infty}_{b}(\mu)$ be the space of all $L^{\infty}(\mu)$ functions with bounded support. A linear operator $T_{\theta}$ is called a $\theta$-type  Calder\'{o}n-Zygmund operator with kernel $K(x,y)$ satisfying (3.2) and (3.3). Moreover, $T_{\theta}$ can be extended to a bounded linear operator on $L^{2}(X)$,
\begin{eqnarray}
T_{\theta}f(x):=\int_{X}K(x,y)f(y)\mathrm{d}\mu(y)
\end{eqnarray}
for all $f\in L^{\infty}_{b}(\mu)$ and $x\notin\mathrm{supp}(f)$.

Given a locally integrable function $b$ and $\theta$-type  Calder\'{o}n-Zygmund operator $T_{\theta}$ on $X$, the linear commutator $[b,T_{\theta}]$ is defined as£»
\begin{eqnarray}
[b,T_{\theta}]f(x):=b(x)T_{\theta}f(x)-T_{\theta}(bf)(x)=\int_{X}[b(x)-b(y)]K(x,y)f(y)\mathrm{d}\mu(y).
\end{eqnarray}

The main theorems of this section is stated as follows.\\
\begin{theorem}\ \
Let $p\in(1,\infty)$, $\omega\in A_{p}(\mu)$, $\varphi\in\Phi_{p}$. Let $\phi: (0,\infty)\rightarrow(0,\infty)$ be an increasing function, continuous function satisfying conditions (2.4) and (2.5).
Then $T_{\theta}$ defined as in (3.4) is bounded on $\mathcal{L}^{p),\phi}_{\varphi}(\omega)$, that is, there exists a constant $C>0$ such that, for all $f\in\mathcal{L}^{p),\phi}_{\varphi}(\omega)$,
$$
\|T_{\theta}(f)\|_{\mathcal{L}^{p),\phi}_{\varphi}(\omega)}\leq C\|f\|_{\mathcal{L}^{p),\phi}_{\varphi}(\omega)}.
$$
\end{theorem}
\begin{theorem}\ \
Let $p\in(1,\infty)$, $\omega\in A_{p}(\mu)$, $b\in \mathrm{BMO(\mu)}$, $\varphi\in\Phi_{p}$. Let $\phi: (0,\infty)\rightarrow(0,\infty)$ be an increasing function, continuous function satisfying conditions (2.4) and (2.5).
Then $[b,T_{\theta}]$ defined as in (3.5) is bounded on $\mathcal{L}^{p),\phi}_{\varphi}(\omega)$, that is, there exists a constant $C>0$ such that, for all $f\in\mathcal{L}^{p),\phi}_{\varphi}(\omega)$,
$$
\|[b,T_{\theta}]f\|_{\mathcal{L}^{p),\phi}_{\varphi}(\omega)}\leq C\|b\|_{BMO(\mu)}\|f\|_{\mathcal{L}^{p),\phi}_{\varphi}(\omega)}.
$$
\end{theorem}
To formulate the above theorems we also need the following lemma.
\begin{lemma}\rm
Let $p\in(1,\infty)$, $\omega\in A_{p}(\mu)$. Let $\phi: (0,\infty)\rightarrow(0,\infty)$ be an increasing function, continuous function satisfying conditions (2.4) and (2.5) and $\theta$ be a non-negative, non-decreasing function on $(0,\infty)$ with satisfying (3.1).
Then $T_{\theta}$ defined as in (3.4) is bounded on $\mathcal{L}^{p,\phi}(\omega)$, that is, there exists a constant $C>0$ such that, for all $f\in\mathcal{L}^{p,\phi}(\omega)$,
$$
\|T_{\theta}(f)\|_{\mathcal{L}^{p,\phi}(\omega)}\leq C\|f\|_{\mathcal{L}^{p,\phi}(\omega)}.
$$
\end{lemma}
\noindent
\begin{proof}[\textbf{Proof.}]
Let $p\in(1,\infty)$, we only need to consider that for any fixed ball $B=B(x_0,r)\subset X$,
\begin{eqnarray}
\bigg\{\frac{1}{\phi(\omega(B))}\int_{B}[T_{\theta}f(x)]^{p}\omega(x)\mathrm{d}\mu(x)\bigg\}^{\frac{1}{p}}\leq C\|f\|_{\mathcal{L}^{p,\phi}(\omega)}.
\end{eqnarray}
To estimate (3.6), we decompose $f$ as $f:=f_{1}+f_{2}$, where $f_{1}:=f\chi_{2B}$ and $2B=B(x_0,2r)$, write
\begin{eqnarray*}
&&\bigg\{\frac{1}{\phi(\omega(B))}\int_{B}[T_{\theta}(f)(x)]^{p}\omega(x)\mathrm{d}\mu(x)\bigg\}^{\frac{1}{p}}\\
&&\quad\quad \leq \bigg\{\frac{1}{\phi(\omega(B))}\int_{B}[T_{\theta}(f_{1})(x)]^{p}\omega(x)\mathrm{d}\mu(x)\bigg\}^{\frac{1}{p}}
+\bigg\{\frac{1}{\phi(\omega(B))}\int_{B}[T_{\theta}(f_2)(x)]^{p}\omega(x)\mathrm{d}\mu(x)\bigg\}^{\frac{1}{p}}\\
&&\quad\quad ={G}_{1}+{G}_{2}.
\end{eqnarray*}
The estimate for $G_1$ goes as follows. From [\cite{D}, Theorem 1.3] slightly modified, we know that the $T_{\theta}$ is bounded on $L^{p}(\omega)$ for $p\in(1,\infty)$.
By applying (2.3) and (2.4), implies that
\begin{eqnarray*}
G_1&\leq&\frac{1}{[\phi(\omega(B))]^{\frac{1}{p}}}\bigg[\int_{X}|f_{1}(x)|^{p}\omega(x)\mathrm{d}\mu(x)\bigg]^{\frac{1}{p}}\\
&\leq&C\bigg[\frac{1}{\phi(\omega(2B))}\int_{2B}|f(x)|^{p}\omega(x)\mathrm{d}\mu(x)\bigg]^{\frac{1}{p}}\bigg[\frac{\phi(\omega(2B))}{\phi(\omega(B))}\bigg]^{\frac{1}{p}}\\
&\leq&C\|f\|_{\mathcal{L}^{p,\phi}(\omega)}\bigg[\frac{\omega(2B)}{\omega(B)}\bigg]^{\frac{1}{p}}\\
&\leq&C\|f\|_{\mathcal{L}^{p,\phi}(\omega)}.
\end{eqnarray*}
For term $G_2$, notice that, for any $x\in B$ and $y\in(2B)^{c}$, we obtain $d(x,y)\sim d(x_0,y)$ and $V(x,y)\sim V(x_0,y)$, by virtue of H\"{o}lder inequality and {\bf Definition 1.1} and {\bf Lemma 2.4} ,
\begin{eqnarray*}
|T_{\theta}(f_2)(x)|&\leq&\int_{d(y,x_0)\geq 2r}|K(x,y)f(y)|\mathrm{d}\mu(y)\\
&\leq&C\int_{d(y,x_0)\geq 2r}\frac{|f(y)|}{V(x,y)}\mathrm{d}\mu(y)\\
&\sim&C\int_{d(y,x_0)\geq 2r}\frac{|f(y)|}{V(x_0,y)}\mathrm{d}\mu(y)\\
&\leq&C\sum^{\infty}_{k=1}\int_{2^{k}r\leq d(y,x_0)\leq2^{k+1}r}\frac{|f(y)|}{V(x_0,y)}\mathrm{d}\mu(y)\\
&\leq&C\sum^{\infty}_{k=1}\frac{1}{V_{2^{k}r}(x_0)}\bigg[\int_{B(x_0,2^{k+1}r)}|f(y)|^{p}\omega(y)\mathrm{d}\mu(y)\bigg]^{\frac{1}{p}}
\bigg[\int_{B(x_0,2^{k+1}r)}\omega(y)^{1-p^{\prime}}\mathrm{d}\mu(y)\bigg]^{\frac{1}{p^{\prime}}}\\
&\leq&C\sum^{\infty}_{k=1}\frac{[\phi(\omega(B(x_0,2^{k+1}r)))]^{\frac{1}{p}}}{V_{2^{k}r}(x_0)}
\cdot\frac{V_{2^{k+1}r}(x_0)}{[\omega(B(x_0,2^{k+1}r))]^{\frac{1}{p}}}\|f\|_{\mathcal{L}^{p,\phi}(\omega)}\\
&\leq&C\bigg[\frac{\phi(\omega(B))}{\omega(B)}\bigg]^{\frac{1}{p}}\|f\|_{\mathcal{L}^{p,\phi}(\omega)}.
\end{eqnarray*}
Thus
\begin{eqnarray*}
\bigg\{\frac{1}{\phi(\omega(B))}\int_{B}[T_{\theta}(f_2)(x)]^{p}\omega(x)\mathrm{d}\mu(x)\bigg\}^{\frac{1}{p}}
&\leq&C \bigg[\frac{\phi(\omega(B))}{\omega(B)}\bigg]^{\frac{1}{p}}\bigg[\frac{\omega(B)}{\phi(\omega(B))}\bigg]^{\frac{1}{p}}\|f\|_{\mathcal{L}^{p,\phi}(\omega)}\\
&\leq&C \|f\|_{\mathcal{L}^{p,\phi}(\omega)}.
\end{eqnarray*}
Which, together with estimate of $\mathrm{G}_{1}$, we obtain the desired result.
\end{proof}
\begin{lemma}\rm \cite{LLW}
Let $p\in(1,\infty)$, $\omega\in A_{p}(\mu)$ and $b\in \mathrm{BMO(\mu)}$. Let $\phi: (0,\infty)\rightarrow(0,\infty)$ be an increasing function, continuous function satisfying conditions (2.4) and (2.5) and $\theta$ be a non-negative, non-decreasing function on $(0,\infty)$ with satisfying (3.1). Then the commutator $[b,T_{\theta}]$ defined as in (3.5) is bounded on $\mathcal{L}^{p,\phi}(\omega)$.
\end{lemma}

\begin{proof}[\textbf{Proof of Theorem} $\mathrm{3.1}$]
Let $\delta$ be a fixed constant satisfying $0<\varepsilon<\delta<p-1$. By applying {\bf Definition 1.1}, observe that
\begin{eqnarray*}
\|T_{\theta}(f)\|_{\mathcal{L}^{p),\phi}_{\varphi}(\omega)}&=& \mathop{\sup}\limits_{0<\varepsilon<p-1}\varphi(\varepsilon)\|T_{\theta}(f)\|_{\mathcal{L}^{p-\varepsilon,\phi}(\omega)}\\
&\leq&\mathop{\sup}\limits_{0<\varepsilon<\delta}\varphi(\varepsilon)\|T_{\theta}(f)\|_{\mathcal{L}^{p-\varepsilon,\phi}(\omega)}
+\mathop{\sup}\limits_{\delta<\varepsilon<p-1}\varphi(\varepsilon)\|T_{\theta}(f)\|_{\mathcal{L}^{p-\varepsilon,\phi}(\omega)}\\
&=&\mathrm{H}_{1}+\mathrm{H}_{2}.
\end{eqnarray*}
The estimates for $\mathrm{H}_{1}$ goes as follows. From {\bf Definition 1.1} and {\bf Lemma 3.1}, it follows that
\begin{eqnarray*}
\mathrm{H}_{1}&=&\mathop{\sup}\limits_{0<\varepsilon<\delta}\varphi(\varepsilon)\|T_{\theta}(f)\|_{\mathcal{L}^{p-\varepsilon,\phi}(\omega)}\\
&\leq& C\mathop{\sup}\limits_{0<\varepsilon<\delta}\varphi(\varepsilon)\|f\|_{\mathcal{L}^{p-\varepsilon,\phi}(\omega)}\\
&\leq& C\|f\|_{\mathcal{L}^{p),\phi}_{\varphi}(\omega)}.
\end{eqnarray*}
Fix $\varepsilon\in(\delta, p-1)$ so that $\frac{p-\delta}{p-\varepsilon}>1$. Using H\"{o}lder inequality with respect to the $\big(\frac{p-\delta}{p-\varepsilon}\big)^{\prime}=\frac{p-\delta}{\varepsilon-\delta}$ and the boundedness of $T_{\theta}$ in $L^{p}(\omega)$ for $p\in(1,\infty)$, we can deduce that
\begin{eqnarray*}
&&\mathop{\sup}\limits_{\delta<\varepsilon<p-1}\varphi(\varepsilon)\|T_{\theta}(f)\|_{\mathcal{L}^{p-\varepsilon,\phi}(\omega)}\\
&&\quad\quad=\mathop{\sup}\limits_{\delta<\varepsilon<p-1}\varphi(\varepsilon)\mathop{\sup}\limits_{B\subset X}
[\phi(\omega(B))]^{-\frac{1}{p-\varepsilon}}\bigg(\int_{B}|T_{\theta}(f)(x)|^{p-\varepsilon}\omega(x)\mathrm{d}\mu(x)\bigg)^{\frac{1}{p-\varepsilon}}\\
&&\quad\quad\leq \mathop{\sup}\limits_{\delta<\varepsilon<p-1}\varphi(\varepsilon)\mathop{\sup}\limits_{B\subset X}[\phi(\omega(B))]^{-\frac{1}{p-\varepsilon}}
\bigg(\int_{B}|T_{\theta}(f)(x)|^{p-\delta}\omega(x)\mathrm{d}\mu(x)\bigg)^{\frac{1}{p-\delta}}\omega(B)^{\frac{\varepsilon-\delta}{(p-\delta)(p-\varepsilon)}}\\
&&\quad\quad\leq C\mathop{\sup}\limits_{\delta<\varepsilon<p-1}\varphi(\varepsilon)\mathop{\sup}\limits_{B\subset X}[\phi(\omega(B))]^{-\frac{1}{p-\varepsilon}}
\bigg(\int_{B}|f(x)|^{p-\delta}\omega(x)\mathrm{d}\mu(x)\bigg)^{\frac{1}{p-\delta}}\omega(B)^{\frac{\varepsilon-\delta}{(p-\delta)(p-\varepsilon)}}\\
&&\quad\quad\leq C\mathop{\sup}\limits_{\delta<\varepsilon<p-1}\varphi(\varepsilon)\mathop{\sup}\limits_{B\subset X}[\phi(\omega(B))]^{-\frac{1}{p-\varepsilon}}[\phi(\omega(B))]^{\frac{1}{p-\delta}}\omega(B)^{\frac{\varepsilon-\delta}{(p-\delta)(p-\varepsilon)}}\\
&&\quad\quad\quad\times[\phi(\omega(B))]^{-\frac{1}{p-\delta}}\bigg(\int_{B}|f(x)|^{p-\delta}\omega(x)\mathrm{d}\mu(x)\bigg)^{\frac{1}{p-\delta}}.\\
\end{eqnarray*}
Let $$S=[\phi(\omega(B))]^{-\frac{1}{p-\varepsilon}}[\phi(\omega(B))]^{\frac{1}{p-\delta}}\omega(B)^{\frac{\varepsilon-\delta}{(p-\delta)(p-\varepsilon)}}.$$
Since $\delta<p-1$ and $\varepsilon\in(\delta,p-1)$, imply that $$0<\frac{\varepsilon-\delta}{(p-\delta)(p-\varepsilon)}<\frac{p-1-\delta}{p-\delta}.$$
By applying the monotonicity of $\phi$, we can deduce that
\begin{eqnarray*}
S&\leq &[\phi(\omega(B))]^{-\frac{1}{p-\delta}}[\phi(\omega(B))]^{\frac{1}{p-\delta}}\omega(B)^{\frac{\varepsilon-\delta}{(p-\delta)(p-\varepsilon)}}\\
&\leq& \omega(B)^{\frac{p-1-\delta}{(p-\delta)(p-\delta)}}\leq C.
\end{eqnarray*}

Combing above the estimate, we further obtain that
\begin{eqnarray*}
&&\mathop{\sup}\limits_{\delta<\varepsilon<p-1}\varphi(\varepsilon)\|T_{\theta}(f)\|_{\mathcal{L}^{p-\varepsilon,\phi}(\omega)}\\
&&\quad\quad\leq C\mathop{\sup}\limits_{\delta<\varepsilon<p-1}\varphi(\varepsilon)\mathop{\sup}\limits_{B\subset X}[\phi(\omega(B))]^{-\frac{1}{p-\delta}}\bigg(\int_{B}|f(x)|^{p-\delta}\omega(x)\mathrm{d}\mu(x)\bigg)^{\frac{1}{p-\delta}}\\
&&\quad\quad\leq C\mathop{\sup}\limits_{\delta<\varepsilon<p-1}\varphi(\varepsilon)[\phi(\delta)]^{-1}\phi(\delta)\mathop{\sup}\limits_{B\subset X}[\phi(\omega(B))]^{-\frac{1}{p-\delta}}\bigg(\int_{B}|f(x)|^{p-\delta}\omega(x)\mathrm{d}\mu(x)\bigg)^{\frac{1}{p-\delta}}\\
&&\quad\quad\leq C\varphi(p-1)[\phi(\delta)]^{-1}\|f\|_{\mathcal{L}^{p),\phi}_{\varphi}(\omega)}\\
&&\quad\quad \leq C\|f\|_{\mathcal{L}^{p),\phi}_{\varphi}(\omega)}.
\end{eqnarray*}
Which, together with estimate of $\mathrm{H}_{1}$, we obtain the desired result.
\end{proof}
\begin{proof}[\textbf{Proof of Theorem} $\mathrm{3.2}$]
First observe that the boundedness of  $[b,T_{\theta}]$ on $L^{p}(\omega)$ for $\omega\in A_{p}(\mu)$ (see \cite{D}), and the Calder\'on-Zygmund interpolation theorem imply that there is a number $\delta$, $\delta\in(0,p-1)$, such that
$$\|[b,T_{\theta}](f)\|_{\mathcal{L}^{p-\varepsilon,\phi}(\omega)}\leq C\|b\|_{BMO(\mu)}\|f\|_{\mathcal{L}^{p-\varepsilon,\phi}(\omega)}, \quad \varepsilon\in (0,\delta].$$
Fix $\varepsilon\in(\delta, p-1)$ so that $\frac{p-\delta}{p-\varepsilon}>1$, by virtue of H\"{o}lder's inequality and {\bf Lemma 3.2}, observe that
\begin{eqnarray*}
&&\|[b,T_{\theta}](f)\|_{\mathcal{L}^{p),\phi}_{\varphi}(\omega)}=\max\bigg\{
\mathop{\sup}\limits_{0<\varepsilon<\delta}\varphi(\varepsilon)\|[b,T_{\theta}](f)\|_{\mathcal{L}^{p-\varepsilon,\phi}(\omega)},\mathop{\sup}\limits_{\delta<\varepsilon<p-1}\varphi(\varepsilon)\|[b,T_{\theta}](f)\|_{\mathcal{L}^{p-\varepsilon,\phi}(\omega)}\bigg\}\\
&&\leq \max\bigg\{\mathop{\sup}\limits_{0<\varepsilon<\delta}\varphi(\varepsilon)\|[b,T_{\theta}](f)\|_{\mathcal{L}^{p-\varepsilon,\phi}(\omega)},\\
&&\quad\quad\quad\quad\mathop{\sup}\limits_{\delta<\varepsilon<p-1}\varphi(\varepsilon)\mathop{\sup}\limits_{B\subset X}[\phi(\omega(B))]^{-\frac{1}{p-\varepsilon}}
\|[b,T_{\theta}](f)\|_{L^{p-\delta}(\omega)}\omega(B)^{\frac{\varepsilon-\delta}{(p-\delta)(p-\varepsilon)}}\bigg\}\\
&&\leq \max\bigg\{\mathop{\sup}\limits_{0<\varepsilon<\delta}\varphi(\varepsilon)\|[b,T_{\theta}](f)\|_{\mathcal{L}^{p-\varepsilon,\phi}(\omega)},\\
&&\quad\quad\quad\quad\mathop{\sup}\limits_{\delta<\varepsilon<p-1}\varphi(\varepsilon)\mathop{\sup}\limits_{B\subset X}
[\phi(\omega(B))]^{-\frac{1}{p-\delta}}\|[b,T_{\theta}](f)\|_{L^{p-\delta}(\omega)}\bigg[\frac{\omega(B)}{\phi(\omega(B))}\bigg]^{\frac{\varepsilon-\delta}{(p-\delta)(p-\varepsilon)}}\bigg\}\\
&&\leq \max\bigg\{\mathop{\sup}\limits_{0<\varepsilon<\delta}\varphi(\varepsilon)\|[b,T_{\theta}](f)\|_{\mathcal{L}^{p-\varepsilon,\phi}(\omega)},\\
&&\quad\quad\quad\quad\mathop{\sup}\limits_{\delta<\varepsilon<p-1}\bigg[\frac{\omega(B)}{\phi(\omega(B))}\bigg]^{\frac{\varepsilon-\delta}{(p-\delta)(p-\varepsilon)}}
\mathop{\sup}\limits_{0<\varepsilon<\delta}\varphi(\varepsilon)\|[b,T_{\theta}](f)\|_{\mathcal{L}^{p-\varepsilon,\phi}(\omega)}\bigg\}.
\end{eqnarray*}

Let $S:=\mathop{\sup}\limits_{0<\varepsilon<\delta}\varphi(\varepsilon)\|[b,T_{\theta}](f)\|_{\mathcal{L}^{p-\varepsilon,\phi}(\omega)}$ and
$T:=\mathop{\sup}\limits_{\delta<\varepsilon<p-1}\bigg[\frac{\omega(B)}{\phi(\omega(B))}\bigg]^{\frac{\varepsilon-\delta}{(p-\delta)(p-\varepsilon)}}$.

Then
\begin{eqnarray*}
\|[b,T_{\theta}](f)\|_{\mathcal{L}^{p),\phi}_{\varphi}(\omega)}\leq \max\{1,T\}\cdot S\leq
C\|b\|_{BMO(\mu)}\|f\|_{\mathcal{L}^{p-\varepsilon,\phi}_{\varphi}(\omega)}.
\end{eqnarray*}

Thus, we obtain the desired result.
\end{proof}

\color{black}

\end{document}